\documentclass{article}

\usepackage{arxiv}

\usepackage[T1]{fontenc}
\usepackage[utf8]{inputenc}
\usepackage{tocloft,amsmath,amsfonts,amssymb,amsthm,mathtools,dsfont,bbold,graphicx,fancybox,enumitem,hyperref}

\usepackage{caption}
\usepackage{subcaption}
\usepackage{booktabs}       

\usepackage{tikz}
\usetikzlibrary{quotes,angles}

\newtheorem{Theorem}{Theorem}[section]

\newtheorem{Lemma}[Theorem]{Lemma}
\newtheorem{Corollary}[Theorem]{Corollary}

\newtheorem{Remark}[Theorem]{Remark}
\newtheorem{example}[Theorem]{Example}

\usepackage{algorithmic}
\usepackage[ruled,linesnumbered]{algorithm2e}

\newcommand{\N}{\mathbb{N}}
\newcommand{\R}{\mathbb{R}}
\renewcommand{\P}{\mathbb{P}}

\newcommand{\prob}[6]{\mathbb{P}(#5_{#3}#6#4\vert X_1=#2,X_0=#1)}
\newcommand{\hattau}[3]{\wt\tau_{#1,#2\to#3}}
\newcommand{\hatphi}[3]{{\varphi}_{#1,#2\to#3}}
\newcommand{\hatphitwo}[2]{{\varphi}_{#1\to#2}}

\renewcommand{\Delta}{\mathrm{deg}}
\DeclareMathAlphabet{\mathbbx}{U}{bbold}{m}{n}
\newcommand{\oneb}{\mathbbx{1}} 
\def\G{\mathcal{G}}
\newcommand\sou{{\tt S}}
\newcommand\ter{{\tt T}}
\newcommand\inp{\mathcal{I}}
\newcommand\out{\mathcal{O}}
\newcommand\wh{\widehat}
\newcommand\wt{\widetilde}

\newcommand{\uno}{\oneb} 

\title{Hitting times for second-order random walks}

\author{%
    \textbf{Dario Fasino}, \textbf{Arianna Tonetto}\\
    Department of Mathematics, Computer Science and Physics\\ University of Udine\\ 33100, Udine (Italy)\\ \texttt{dario.fasino@uniud.it} 
    \And
    \textbf{Francesco Tudisco}
    \\
    School of Mathematics\\
    Gran Sasso Science Institute\\
    67100, L'Aquila (Italy) \\
    \texttt{francesco.tudisco@gssi.it} 
}

\begin{document}

\maketitle

\begin{abstract}
A second-order random walk on a graph or network is a random walk where transition probabilities depend not only on the present node but also on the previous one. A notable example is the non-backtracking random walk, where the walker is not allowed to revisit a node in one step.
Second-order random walks
can model physical diffusion phenomena in a more realistic way than traditional random walks and have been very successfully used in various network mining and machine learning settings. However, numerous questions are still open for this type of stochastic processes. In this work we extend well-known results concerning mean hitting and return times of standard random walks to the second-order case. In particular, we provide simple formulas that allow us to compute these numbers by solving suitable systems of linear equations. Moreover, by introducing the ``pullback'' first-order stochastic process of a second-order random walk, we provide second-order versions of the renowned Kac's and random target lemmas. 
\end{abstract}


\section{Introduction}

Random walks on graphs are a key concept in network theory used to model various dynamics such as user navigation and epidemic spreading \cite{newmanbook}, to quantify node centrality \cite{newman,noh}, or to reveal communities and core-periphery structures \cite{cucuringu,dellarossa,rombach,tudisco2018core}. The standard random walk on a graph considers a hypothetical walker that starting from node $i$ moves to a new node of the graph choosing it uniformly at random among the neighbors of $i$.
This defines a memoryless stochastic process which corresponds to a Markov chain on the nodes. However, the process described by a standard graph random walk  is often 
localized and may fail to capture the underlying physical model as well as to fully exploit the global graph structure.  For this reason, 
there is a growing interest in recent years 
towards so-called higher-order stochastic processes whose evolution may depend on past states, including stochastic processes modeled by transition probability tensors \cite{benson2017spacey,fasino2019higher} as well as non-local diffusion mappings \cite{cipolla2020nonlocal, estrada2017random} and non-backtracking random walks \cite{fitzner}.

The latter is one of the best known examples of second-order stochastic processes, where node transitions depend on both the current and the previous state spaces as the process is not allowed to \textit{backtrack}, i.e.\ to immediately revisit the previous space in the next step.
Non-backtracking random walks are often more appropriate than classical random walks to model real-world diffusive phenomena on networks where a message-passing or disease-spreading analogy is relevant. In particular, non-backtracking walks can avoid unwanted localization effects on the leading eigenvector of the adjacency matrix of a graph \cite{kawamoto,krzakala,martin} and 
their mixing rate, in many cases, is faster than the one of classical random walks \cite{alon,CioabaXu}. Furthermore, many computations with non-backtracking walks have negligible overheads compared to standard approaches \cite{arrigo1,grindrod}.

The non-backtracking paradigm is at the basis of several second-order stochastic processes appearing in recent graph-theoretic, network mining and machine learning literature. 
For example, the mixing time of a clique-based version of the non-backtracking random walk on regular graphs is studied in  \cite{CioabaXu}.
Similarity  measures based on second-order random walks are used in \cite{Wu2016} to measure node proximity and improve classical walk-based link-prediction and clustering algorithms. Eigenvalues and eigenvectors of non-backtracking matrices have been shown to provide remarkable performances in community detection and in various centrality measurements \cite{arrigo1,krzakala,torres2021nonbacktracking}. The graph-embedding generated by the popular {\tt node2vec} algorithm \cite{node2vec} is based on a flexible second-order random walk depending on two parameters: one used to control the likelihood of backtracking (i.e.\ immediately revisiting a node in the walk), the other used to bias the navigation towards or away from the immediately preceding nodes.

In this work we are concerned with hitting and return times of second-order random walks on graphs with possibly countably many vertices. Roughly speaking, a hitting time is the first time at which a  stochastic process reaches a given subset of its state space, while a return time is the first time the process gets back to a given starting state.  
Hitting and return times are well understood for Markov chains, which include as a particular case the stochastic process that describes a standard random walk on a graph. In this case, the state space is given by the set of nodes and explicit formulas are available to compute mean hitting and return times 
by means  of suitable systems linear equations which, for finite and connected graphs, have unique solutions. Moreover, these quantities have numerous applications to, for example,  quantify node distances and similarities. For this reason, they are an important tool to evaluate network cohesion and node centrality and are at the basis of popular node classification and link prediction algorithms~\cite{liben2007link,nadler2009semi}. 

In this paper, we extend several results concerning mean hitting and return times of standard first-order stochastic processes to the second-order setting. In particular, we show that also mean hitting times of second-order random walks can be obtained from a set of linear equations, which has potential applications in higher-order classification and link prediction methods \cite{arrigo2020framework,benson2018simplicial,tudisco2021nonlinear}. 
Moreover, we prove that mean return times of
second-order random walks coincide with standard return times of suitably defined random walks in the original graph. In particular, mean return times for non-backtracking random walks coincide with the usual mean return times in all finite undirected graphs.

The paper is organized as follows. In Section \ref{notations} and \ref{sec:2ndORW} we present the notations used in this work and some basic results about random walks and second-order random walks, respectively. 
In Section \ref{hitting} we introduce mean hitting times and return times for second-order random walks and we present first results concerning possible ways to compute them numerically via the solution of systems of linear equations.
Our main results are in Section \ref{sec:main}.
There, we analyze second-order mean hitting and return times under what we call ``equilibrium assumption'', corresponding to the condition that the second-order random walk is at the steady state. In particular, we derive
an algebraic relationship between the hitting time matrix for a second-order random  walk and the one corresponding to a standard random walk on the directed line graph of the original graph. 
Moreover, we show second-order analogues of the Kac's and the random target lemmas for standard Markov chains.
Finally, in Section \ref{experiments}, we illustrate some of the results of this paper on 
some real-world examples.

\section{Notations and basic results}  \label{notations}

An oriented graph (or network) is a pair $\G=(V,E)$, where $V$ is the vertex (or node) set and $E$ is the set of oriented edges, i.e., ordered pairs of nodes. We say that $(i,j)\in E$ 
is an edge going from $i$ to $j$. 
From here onward, letters $i,j,k$ denote generic nodes of $\G$,
while $e,f$ are used to denote generic edges of $\G$.
Furthermore,
\begin{itemize}
\item
if $e = (i,j)\in E$ then we set $\sou(e) = i$ and $\ter(e) = j$, 
the source and terminal node of the edge $e$;
\item
for $i\in V$ we denote by $\inp_i = \{e\in E : \ter(e) = i\}$
and $\out_i = \{e\in E : \sou(e) = i\}$ the in-neighborhood and out-neighborhood of $i$, respectively.
The in-degree and the out-degree of node $i$ are $d^-_i = |\out_i|$ and $d^+_i = |\inp_i|$, respectively.
\end{itemize}
In the present work, $\G$ can have a countably infinite number of nodes. However, 
we will assume that $\G$ is locally finite, that is, $V$ is possibly infinite but each node has finite in- and out-degree.

A graph $\G = (V,E)$ can be completely described by its (possibly infinite dimensional) adjacency matrix $A=(A_{ij})$, defined as 
$$ 
    A_{ij}=\begin{cases}
    1 & \hbox{ if } (i,j)\in E ;  \\
    0 & \hbox{ otherwise.}
    \end{cases}
$$ 
If the adjacency matrix is symmetric then $\G$ is called undirected. Hence, in an undirected graph each edge $(i,j)$ has its own reciprocal edge $(j,i)$. In that case we have $d^-_i = d^+_i$ and this common value is the degree of node $i$, denoted as $d_i$.

A \emph{walk} of length $m$ on $\G$ is a sequence $v_0,v_1,\ldots,v_m \in V$ of nodes 
such that $(v_{i-1},v_i)\in E$ for $i=1,\ldots,m$. 
If for all $i,j\in V$ there exists a walk from $i$ to $j$, 
then the graph $\G$ is called strongly connected.

Finally, the symbols 
$\P$ and $\mathbb{E}$  denote the probability and the expectation, respectively. All vectors used in the text are column vectors, with possibly infinitely many entries. In particular, the symbol $\oneb$ denotes the column vector with all entries equal to 1, the size of which should be clear from the context.

\subsection{Random walks}   \label{sect:rw}

In this section, we recall from \cite{KemenySnell,markov}
some basic concepts and results in Markov chain theory.
A random walk on $\G$ is a sequence of nodes $v_0,v_1,\ldots,v_k,\ldots$ where $v_{k+1}$ is chosen 
at random between the out-neighbours of $v_k$,
according to specified probabilities  
$p_{i,j}:=\P(v_{k+1} = j | v_k = i)$
such that $p_{i,j} = 0$ if $(i,j)\notin E$ and $\sum_{j\in\out_i} p_{i,j} = 1$ for every $i\in V$. The last condition implies that $d^+_i \geq 1$ for every $i\in V$, which will be  assumed henceforth.
A standard choice for $p_{i,j}$ is 
$p_{i,j} = 1/d^+_i$, for $j\in \out_i$. In this case the random walk is called uniform. Our subsequent discussion is not restricted to this case.

A random walk can be viewed as a Markov chain $\{Y_n,n\in\N\}$ with state space $V$ and transition matrix $P$ with entries $P_{ij}=p_{i,j}$. Clearly, $P$ is a row-stochastic matrix. The random variable $Y_n$ represents the state of the walker at time $n$. 

For a fixed subset $S\subset V$ we call \emph{hitting time} to $S$ the random variable $T_S$ defined as $T_S=\min\{n\in\N:Y_n\in S\}$. In other words, $T_S$ is the first time the random walker reaches the subset $S$, starting from a given initial state $Y_0$. Based on $T_S$, we can define the {\emph{hitting probability}} $\varphi_{i\to S}$ from $i$ to $S$ as:
\begin{equation}   \label{eq:hittingpb}
   \varphi_{i\to S}= 
   \P(T_S < + \infty\vert Y_0=i)  ,
\end{equation}
i.e., the probability that, starting from node $i$, the random walker ever hits $S$. 
When $\varphi_{i\to S}=1$, it is interesting to compute the average time taken by the random walker 
to reach $S$, starting from $i$. This quantity is given by 
$$
   \tau_{i\to S}=\mathbb{E}(T_S\vert Y_0=i)
$$
and is called 
the \emph{mean hitting time}
from $i$ to $S$. 
It is well-known that the quantities $\varphi_{i\to S}$ and $\tau_{i\to S}$ can be obtained from the solution of certain linear systems, as shown by the next two theorems, whose proofs can be found in, e.g., \cite{markov}.

\begin{Theorem}  \label{theomarkovphi}
For a graph $\G=(V,E)$ and for a subset $S\subset V$, the hitting probabilities $\{\varphi_{i\to S}:i\in V\}$ are the minimal non-negative solution of the following linear system:
\begin{equation}  \label{markovsys1}
	\begin{cases}
	\varphi_{i\to S}=1 & \hbox{if $i\in S$}\\
	\varphi_{i\to S}=\sum_{j\in V} P_{i,j}\varphi_{j\to S} & \hbox{if $i\notin S$.}
	\end{cases}
\end{equation}
\end{Theorem}

In the previous theorem, minimality of the solution means that if $\{x_i:i\in V\}$ is another non-negative solution of \eqref{markovsys1} then $\varphi_{i\to S}\leq x_i$ for all $i\in V$.

\begin{Theorem} \label{theomarkovtau}
For a graph $\G=(V,E)$ and for a subset $S\subset V$, suppose that $\varphi_{i\to S}=1$ for all $i\in V$. Then the mean hitting times $\{\tau_{i\to S}:i\in V\}$ are the minimal non-negative solution of the following linear system:
\begin{equation}  \label{markovsys2}
	\begin{cases}
	\tau_{i\to S}=0 & \hbox{if } i\in S \\
	\tau_{i\to S}=1+\sum_{j\in V} 
	P_{i,j}\tau_{j\to S} 
	& \hbox{if } i\notin S .
	\end{cases}
\end{equation}
\end{Theorem}

The case of finite graphs is the most usual in applications. In this case, a well-known result shows that 
for strongly connected graphs \eqref{markovsys2} has a unique solution, see e.g., Chap.\ VI of \cite{KemenySnell}. Precisely:

\begin{Theorem}    \label{uniquesolA} 
Let $\G=(V,E)$ be finite and strongly connected. For any given subset $S\subset V$, 
all the hitting probabilities to $S$ are equal to $1$ and the linear system \eqref{markovsys2} has a unique solution.
\end{Theorem}

It is interesting to note that Equation \eqref{markovsys2} can be written in compact matrix-vector form as follows.
Let $t^S$ and $r^S$ be the vectors entrywise defined as follows:
\begin{equation}  \label{eq:tSrS}
   t^S_i = \begin{cases} 0 & i\in S \\
   \tau_{i\to S} & i\notin S   \end{cases}   \qquad
   r^S_i = \begin{cases} \tau^+_{i\to S} & i\in S \\
   0 & i\notin S ,  \end{cases}
\end{equation}
where the numbers $\tau^+_{i\to S}$ are defined as
\begin{equation}\label{eq:tau+}
    \tau^+_{i\to S} = 1 + \sum_{j\in V}P_{i,j}\tau_{j\to S} ,
    \qquad (i\in S) \, .
\end{equation}
 Then, the vector $t^S$  solves the singular equation 
\begin{equation}   \label{eq:tcolumn}
   (I-P)x = \oneb - r^S . 
\end{equation}
In particular, among the infinitely many solutions $x$ of this equation, the vector $t^S$ is the one characterized by the condition $t^S_i = 0$ for $i\in S$.
The quantities \eqref{eq:tau+} are called \textit{mean return times} and represent the average number of steps required to return to $S$ after leaving it from $i\in S$. Indeed, if we consider the random variable $T^+_S = \min \{n \geq 1 : Y_n \in S\}$, then it holds $\tau^+_{i\to S}= \mathbb{E}(T^+_S\vert Y_0=i)$.
The special case when $S$ is a singleton deserves a special attention, and we adopt the simplified notation $\tau_i = \mathbb{E}(T^+_i\vert Y_0=i)$. This quantity is called \textit{mean return time to $i$.}


When $\G$ is finite and strongly connected, by means of the Perron--Frobenius theorem (see, e.g., \cite{meyer}) we can conclude that the random walk $\{Y_n,n\in\N\}$ admits a unique positive invariant distribution, i.e., a vector $\pi>0$ such that 
$\pi^T\oneb = 1$ and $\pi^T P=\pi^T$. 
More generally, when  $V$ is countable, the same conclusion holds true 
if and only if the associated Markov chain is irreducible and positive recurrent, that is, $\tau_i$ is finite for at least one $i\in V$, see e.g., Thm.\ 21.12 of \cite{levin} or Thm.\ 1.7.7 of \cite{markov}. The latter property also implies that all hitting probabilities \eqref{eq:hittingpb} are equal to $1$ and all hitting times are finite. 

With the help of the invariant distribution we can obtain an explicit formula for the mean return time to a set $S \subset V$,
\begin{align*}   
   \tau_S & = 
   \mathbb{E}(T^+_S | Y_0 \in S) \nonumber \\
   & = \sum_{i\in S} \mathbb{P}(Y_0 = i | Y_0 \in S)\mathbb{E}(T^+_S | Y_0 = i)
   = \frac{1}{\sum_{j\in S} \pi_j}
   \sum_{i\in S} \pi_i \tau^+_{i\to S} .
\end{align*}
This number quantifies the average number of steps taken by a random walker initially placed in $S$ to hit again $S$, given that its probability of being in node $i$ is proportional to $\pi_i$.
Indeed, the number $\pi_i/\sum_{j\in S} \pi_j$ gives the conditional probability 
that the random walker is in $i$ given that it is in $S$. The average return time $\tau_S$ can be computed 
by means of Kac's lemma, see e.g., \cite{kac} or~\cite[Lemma 21.13]{levin}:

\begin{Lemma}   \label{lem:kac}
Let $\pi$ be the stationary density of an irreducible Markov chain on the state space $V$. Then for any $S\subset V$,
$$
    \tau_S=\frac{1}{\sum_{i\in S} \pi_i}.
$$
In particular, $\tau_i = 1/\pi_i$ for every $i\in V$.
\end{Lemma}

If $\G$ is finite then the mean hitting times fulfill the so-called random target lemma, see e.g., Lemma 10.1 of \cite{levin}.

\begin{Lemma}  \label{lem:Kemeny}
Let $P$ be irreducible, and let $\pi$ be its stationary density. Then, there exists a constant $\kappa$ such that  
$$
   \sum_{j\in V} \pi_j \tau_{i\to j} = \kappa ,
$$
independently on $i\in V$.
\end{Lemma}
The number $\kappa = \kappa(P)$ appearing in the preceding Lemma is the renowned Kemeny's constant of $P$. This constant 
quantifies the expected number of steps to get from node $i$ to a node $j$, selected randomly according to the stationary density.

Introducing the matrix $T = (\tau_{i\to j})$, whose $ij$-th entry is the mean hitting time from $i$ to $j$, 
the claim of Lemma \ref{lem:Kemeny} can be compactly 
expressed via the identity $T\pi = \kappa \uno$. Since the 
$j$-th column of $T$ fulfills equation \eqref{eq:tcolumn} with $S = \{j\}$, using Lemma \ref{lem:kac}
it is not difficult to verify that 
this matrix solves the equation
\begin{equation}   \label{eq:Tmatrix}
   (I-P)T = \uno\uno^T - \mathrm{Diag}(\pi)^{-1} ,
\end{equation}
where $\mathrm{Diag}(\pi)$ is the diagonal matrix with diagonal entries given by the entries of the vector $\pi$. 
Actually, the matrix $T$ is the  unique solution of \eqref{eq:Tmatrix} with zero diagonal entries.

We conclude this review on random walks with a technical lemma of independent interest and that will be useful for later.
This result shows that, apart of a constant term, the vector $t^S$ defined in \eqref{eq:tSrS} can be expressed by a convex linear combination of the vectors $\{t^i, i\in S\}$. Hence, mean access times to a subset can be computed from mean access times to the single elements of the set.

\begin{Lemma}   \label{lem:1}
For any fixed subset $S\subset V$
there exist positive coefficients 
$\beta$ and  $\{\alpha_i\}$, 
for $i\in S$, 
such that $\sum_{i\in S} \alpha_i = 1$, and
$$
   t^S = \sum_{i\in S} \alpha_i t^i - \beta \uno .
$$
In particular, $\tau_{i\to S} = \sum_{j\in S} \alpha_j\tau_{i\to j} - \beta$ for every $i\in V$.
\end{Lemma}

\begin{proof}
Consider the vector $a = (\alpha_i)_{i\in V}$
with entries 
\begin{equation}   \label{eq:alpha}
   \alpha_i = \begin{cases} 0 & i\notin S \\
   \pi_i\tau^+_{i\to S} & i\in S . \end{cases}
\end{equation}
By \eqref{eq:tcolumn} we have
$$
   0 = \pi^T(I - P)t^S = \pi^T(\uno - r^S) = 
   1 - \sum_{i\in S} \pi_i\tau^+_{i\to S} .
$$
Hence $\sum_i\alpha_i = \sum_{i\in S}\pi_i\tau^+_{i\to S} = 1$. Moreover, from \eqref{eq:Tmatrix} we derive
$$ 
   (I-P) T a = (\uno\uno^T - \mathrm{Diag}(\pi)^{-1}) a
   = \uno (\uno^Ta) - r^S = \uno - r^S .
$$
This proves that the difference $Ta - t^S$ 
belongs to the kernel of $I-P$,
that is, $\sum_{i\in S} \alpha_i  t^i - t^S = \beta \uno$ for some $\beta\in \R$.
Considering the $k$-th entry of this identity
for $k\in S$, we have
$$
   \beta = \sum_{i\in S} \alpha_i  \tau_{k\to i} - t^S_k = 
   \sum_{i\in S} \alpha_i  \tau_{k\to i} > 0 ,
$$
and the proof is complete. 
\end{proof}

\section{Second-order random walks}
\label{sec:2ndORW}

A second-order random walk on a graph $\G = (V,E)$ is a 
walk $v_0,v_1,\ldots,v_k,\ldots$ where the probability of choosing $v_{k+1}$ depends not only on $v_k$ but also on $v_{k-1}$.
More precisely, we describe this random walk by means of a stochastic process $\{X_n, n \in \mathbb{N}\}$, where $X_n \in V$ represents the node where the walker is located at time $n$. Hence, there exists constant probabilities $p_{i,j,k}$ such that, for $n\geq 0$  
\begin{equation}   \label{eq:X}
   \P(X_{n+2} = k | X_{n+1} = j , X_n = i ) = p_{i,j,k} .
\end{equation}
We additionally assume that there are specific probabilities $p'_{i,j}$ for the first transition, that is,
\begin{equation}   \label{eq:X01}
   \P(X_{1} = j | X_{0} = i ) = p'_{i,j}  \qquad (i,j) \in E .
\end{equation}
A second-order random walk is not a Markov chain because it does not fulfill the Markov condition, as we need to remember the previous step in order to take the next step. However, it can be turned into a Markov chain by changing the state space from the vertices of the graph to the directed edges.  
More precisely, let $\wh\G = (\wh V,\wh E)$ be the (directed) line graph of $\G = (V,E)$, where
$\wh V = E$ and $(e,f)\in \wh E$ if and only if $\ter(e) = \sou(f)$.
Then, the second-order random walk can be seen as a (first-order) walk on $\wh\G$, that is, a sequence of directed edges $e_0,e_1,\ldots,e_k,\ldots$ where
$v_0 = \sou(e_0)$ and $v_i = \sou(e_i) = \ter(e_{i-1})$, for $n \geq 1$.
However, we must keep in mind that a 
walk of length $m$ in $\G$ corresponds to a 
walk of length $m-1$ in $\wh\G$. 

For $n\geq 0$ consider the random variable $W_n \in E$ defined as
\begin{equation}   \label{eq:defW}
   W_n = (i,j) \ \Longleftrightarrow 
   \ X_{n+1} = j, \  X_n = i .
\end{equation}
On the basis of \eqref{eq:X} and \eqref{eq:X01}, 
the sequence $\{W_n\}$ is a Markov chain with initial distribution $\mathbb{P}(W_0 = (i,j)) = \mathbb{P}(X_0 = i)p'_{i,j}$ and transition matrix $\wh P$ given by $\wh P_{(i,j)(j,k)} = p_{i,j,k}$.
Hence, if $(e,f)\in \wh E$ then $\wh P_{ef} \geq 0$ 
while $\wh P_{ef} = 0$ if $(e,f)\notin \wh E$.
Note that we allow the case $\wh P_{ef} = 0$ for some pair $(e,f)\in \wh E$,
which may occur in some special circumstances,
notably the non-backtracking random walk on $\G$, see Example \ref{ex:nbtrw} below.

\begin{example}   \label{ex:classic}
The uniform random walk on $\wh\G$ is the Markov chain governed by the transition matrix $\wh P$,
$$
\wh P_{ef} = \begin{cases} 1/d^+_{\sou(f)} & \sou(f) = \ter(e) \\
   0 & \hbox{otherwise.} \end{cases}
$$
As $\wh P_{ef}$ do not depend on $\sou(e)$, 
the corresponding stochastic process on $\G$
is the uniform random walk where 
$P_{ij} = \mathbb{P}(X_{n+1} = j | X_n = i ) = 1/d^+_{i}$.
If $\G$ is undirected, then $\wh P$ is bi-stochastic.
\end{example}
   
\begin{example}   	\label{ex:nbtrw}
The Hashimoto graph of $\G = (V,E)$ is the (directed) graph $\mathcal{H} = (V_{\mathcal{H}},E_{\mathcal{H}})$ 
where $V_{\mathcal{H}} = E$ and $(e,f)\in E_{\mathcal{H}}$ if and only if both
$\ter(e) = \sou(f)$ and $\ter(f) \neq \sou(e)$. 
A random walk on the Hashimoto graph is obviously related to a
non-backtracking random walk on $\G$. We can look at a
random walk on $\mathcal{H}$ as a weighted random walk on the directed line graph $\wh\G = (\wh V,\wh E)$
by applying a null weight on the edges $(e,f) \in \wh E$
such that $\ter(f) = \sou(e)$. The corresponding transition matrix $\wh P$ is defined as follows.
If $e = (i,j)$ and $A$ is the adjacency matrix of $\G$, then
\begin{equation}    \label{eq:phatnbtrw}
   \wh P_{ef} = \begin{cases}
   1/(d^+_j-A_{ji}) & \sou(f) = j \hbox{ and } i \neq \ter(f) \\
   0 & \hbox{otherwise.} \end{cases}
\end{equation}
A directed edge $e = (i,j)\in E$ such that $\out_j = (j,i)$ is called dangling.
The stochasticity condition $\sum_{f\in E} \wh P_{ef} = 1$ implies that no edge in $\G$ is dangling, otherwise the out-degree of the corresponding node in $\mathcal{H}$ would be zero. 
Hence, the non-backtracking random walk on $\G$ is well defined if (and only if) $d^+_i \geq 1$ for $i\in V$ and there are no dangling edges. 
Finally, we recall from e.g., \cite{kempton} that, if $\G$ is undirected then $\wh P$ is bi-stochastic.
\end{example}

\begin{example}   \label{ex:penalized}
Let $\wh P^{(1)}$ and $\wh P^{(0)}$ be the stochastic matrices defined in the Example \ref{ex:classic} and Example \ref{ex:nbtrw}, respectively.
For any $\alpha\in(0,1)$ the matrix 
$\wh P^{(\alpha)} = \alpha \wh P^{(1)} + (1-\alpha) \wh P^{(0)}$
is stochastic (bi-stochastic, if $\G$ is undirected).
The corresponding Markov chain describes random walks in $\G$ where backtracking steps are not completely eliminated, but rather the probability of performing one such step is downweighted by a factor depending on $\alpha$. Thus, the corresponding second-order process can be called a backtrack-downweighted random walk. Combinatorial aspects of this kind of walks are addressed in, e.g., \cite{arrigo3}. 
\end{example}

\begin{Remark}
The Markov chain $\{W_n\}$ defined by \eqref{eq:defW} may not be irreducible even if $\G$ is strongly connected. For example, if $\G$ is an undirected graph consisting of a cycle, then the Hashimoto graph of $\G$ is not connected, and the matrix $\wh P$ built as in Example \ref{ex:nbtrw} is reducible.
However, the matrix $\wh P^{(\alpha)}$ defined in Example \ref{ex:penalized} is irreducible in the sole assumption that $\G$ is strongly connected.
\end{Remark}

\section{Second-order mean hitting times}\label{hitting}

In this section, we introduce the concept of mean hitting time 
for the second-order random walk $\{X_n ,  n \in \N\}$ defined by \eqref{eq:X} and \eqref{eq:X01} and we give results analogous to Theorems \ref{theomarkovphi} and \ref{theomarkovtau} for it. 
For each node $k\in V$, let $\wt T_k$ 
be the random variable counting the first time the process arrives at node $k$,
$$ 
   \wt T_k = \min\{n \geq 0 : X_n=k\} . 
$$ 
The aim of this section is to compute the corresponding \emph{mean second-order hitting time} $\wt{\tau}_{i\to k}=\mathbb{E}(\wt T_k \vert X_0=i)$, i.e., how many steps are required by a second-order random walk on average to reach $k$, starting from $i$. 
To this goal, we first compute the mean second-order hitting times to a node $k$, given both the first and the second node,
\begin{equation}   \label{eq:ttijk}
    \wt{\tau}_{i,j\to k} =
   \mathbb{E}(\wt T_k \vert X_1=j,X_0=i)\, ,
\end{equation}
and then we use these values to compute $\wt{\tau}_{i\to k}$ by means of the formula
\begin{align}
   \wt{\tau}_{i\to k} 
   & = \mathbb{E}(\wt T_k \vert X_0=i)   \nonumber \\
   & = \sum_{j\in V} \P(X_1=j\vert X_0=i)
   \mathbb{E}(\wt T_k\vert X_1=j,X_0=i)
   = \sum_{j\in V} p'_{ij} \wt{\tau}_{i,j\to k} .   \label{eq:hitt2o}
\end{align}
The above formula shows that, unlike mean hitting times for Markov chains, second-order mean hitting times depend on the choice of the first transition probabilities, as defined in \eqref{eq:X01}. 
In the sequel, we will see that there is a quite natural choice for these probabilities. 
On the other hand, the finiteness of the numbers $\wt{\tau}_{i,j\to k}$, which is essential for the finiteness of second-order hitting times, is not influenced by the values of the probabilities $p'_{ij}$.
Indeed, in an infinite network it can happen that $\prob{i}{j}{k}{+\infty}{\wt T}{=}>0$ for some $(i,j)\in E$ and $k\in V$. 
For example, this happens when $\wh P$ is the transition matrix of the uniform random walk on $\wh \G$ and this chain is transient.
In this case, $\hattau{i}{j}{k}=+\infty$ because $\hattau{i}{j}{k}$ is the expectation of a random variable that is equal to $+\infty$ with positive probability. For this reason, we first compute $\varphi_{i,j\to k}=\prob{i}{j}{k}{+\infty}{\wt T}{<}$, the probability that, starting with $X_0=i$ and $X_1=j$, the random walker reaches $k$ in finite time. 
For these quantities, the following result holds:

\begin{Theorem}   \label{theophi}
For all $k\in V$, $\{\hatphi{i}{j}{k}:(i,j)\in E\}$ is the minimal non-negative solution of the following linear system:
\begin{equation}   \label{linearphi}
	\begin{cases}
	\hatphi{i}{j}{k} = 1 & 
	\hbox{if } i = k \text{ or } j = k \\
	\hatphi{i}{j}{k} = \sum_{\ell\in V}p_{i,j,\ell} 
	\hatphi{j}{\ell}{k} & \hbox{otherwise.}
	\end{cases}
\end{equation}
\end{Theorem} 

\begin{proof}
We first show that $\{\hatphi{i}{j}{k}:(i,j)\in E\}$ solves the linear system \eqref{linearphi}. Let us consider an edge $(i,j)\in E$. If $i=k$ then $X_0=k$, so $\wt T_k=0$; if $j=k$ then $X_1=k$ and $\wt T_k=1$. In both cases it immediately follows that $\hatphi{i}{j}{k}=\prob{i}{j}{k}{+\infty}{\wt T}{<}=1$. 
In the other cases, we have 
\begin{align*}
	\hatphi{i}{j}{k} & = 
	\prob{i}{j}{k}{+\infty}{\wt T}{<}  \\
	& = \sum_{\ell\in V} \P(X_2=\ell\vert X_1=j,X_0=i)
	\P(\wt T_k<+\infty\vert X_2=\ell,X_1=j,X_0=i)  \\
	& = \sum_{\ell\in V}p_{i,j,\ell}\hatphi{j}{\ell}{k} ,
\end{align*}
that is, \eqref{linearphi}.
Now let $\{x_{i,j}:(i,j)\in E\}$ be another non-negative solution of \eqref{linearphi} for a fixed $k$. We show that $x_{i,j}\geq\hatphi{i}{j}{k}$ for all $(i,j)\in E$. Indeed, if $i=k$ or $j=k$ then $x_{i,j}=\hatphi{i}{j}{k}$ and the claim follows. 
Otherwise, for $i,j\neq k$ we have $x_{i,j} = 
p_{i,j,k} + \sum_{\ell\neq k} p_{i,j,\ell} x_{j,l}$.
Let $\ell_0 = i$ and $\ell_1 = j$. 
For any $\bar n \geq 2$ we have
\begin{align*}
    x_{\ell_0,\ell_1} & = 
    p_{\ell_0,\ell_1, k} + \sum_{\ell_2\neq k} 
    p_{\ell_0,\ell_1,\ell_2} x_{\ell_1,\ell_2} \\
    & = 
    p_{\ell_0,\ell_1, k} + \sum_{\ell_2\neq k} 
    p_{\ell_0,\ell_1,\ell_2} 
    \bigg[ p_{\ell_1,\ell_2, k} + \sum_{\ell_3\neq k} 
    p_{\ell_1,\ell_2,\ell_3} x_{\ell_2, \ell_3} \Big] \\
    & = \ldots \\
    & = \sum_{n = 1}^{\bar n} \Bigg[ 
    \sum_{\ell_2,\ldots,\ell_n \neq k} 
    \prod_{t = 2}^{n} 
    p_{\ell_{t-2},\ell_{t-1},\ell_n} \Bigg]
    p_{\ell_{n-1},\ell_{n}, k} + 
    \sum_{\ell_2,\ldots,\ell_{\bar n + 1}\neq k} 
    \prod_{t = 2}^{\bar n+1} p_{\ell_{t-2},\ell_{t-1},\ell_{t}} x_{\ell_{t-1}\ell_{t}}  \\
    & \geq \sum_{n = 1}^{\bar n} \Bigg[ 
    \sum_{\ell_2\neq k}\cdots\sum_{\ell_n \neq k} 
    \prod_{t = 2}^{n} 
    p_{\ell_{t-2},\ell_{t-1},\ell_n} \Bigg]
    p_{\ell_{n-1},\ell_{n}, k} ,
\end{align*}
where we used the inequality $x_{\ell_{t-1}\ell_{t}}\geq 0$
and the empty sum (i.e., the square bracket in the $n = 1$ case) must be considered equal to $1$.
Observe that for $n = 1,2,3,\ldots$
$$
    \Bigg[ \sum_{\ell_2\neq k}\cdots\sum_{\ell_n \neq k} 
    \prod_{t = 2}^{n} 
    p_{\ell_{t-2},\ell_{t-1},\ell_n} \Bigg]
    p_{\ell_{n-1},\ell_{n}, k} 
    = \mathbb{P}(\wt T_k = n+1 | X_1 = j, X_0 = i) .
$$
Hence we have $x_{i,j} \geq \mathbb{P}(\wt T_k \leq \bar n+1 | X_1 = j, X_0 = i)$.
Finally,
\begin{align*}
    x_{i,j} & \geq \lim_{\bar n\to+\infty}
    \prob{i}{j}{k}{\bar n + 1}{\wt T}{\leq} \\ &  = \prob{i}{j}{k}{+\infty}{\wt T}{<} 
    = \hatphi{i}{j}{k} 
\end{align*}
and the claim follows. 
\end{proof}

On the basis of the preceding theorem we can compute the probability that the second-order process $\{X_n\}$ hits a node $k\in V$ in a finite number of steps. 

\begin{Corollary}    \label{cor:hitprob}
For any fixed $k\in V$ define $\hatphitwo{i}{k} = \mathbb{P}(\wt T_k < +\infty | X_0 = i )$. Then, the vector $(\hatphitwo{i}{k})_{i\in V}$ is the vector $(\hatphi{i}{j}{k})_{(i,j)\in E}$ premultiplied by  the matrix $M\in\R^{|V|\times |E|}$ given by 
\begin{equation}    \label{eq:M}
   M_{ie} = \begin{cases}
   p'_{i,j} & e = (i,j) \in E \\ 0 & \hbox{else.}
   \end{cases}
\end{equation}
\end{Corollary}

\begin{proof}
By elementary considerations, we have
\begin{align*}
   \hatphitwo{i}{k} & = \mathbb{P}(\wt T_k < +\infty | X_0 = i ) \\
   & = \sum_{j\in V} \P(X_1=j\vert X_0=i)
   \prob{i}{j}{k}{+\infty}{\wt T}{<} = \sum_{j\in V} p'_{i,j} 
   \hatphi{i}{j}{k}. 
\end{align*}
\end{proof}

In a similar way, we can compute the second-order mean hitting times, as shown in the next theorem:
\begin{Theorem}   \label{theotau}
	In the previous notation, if $\hatphi{i}{j}{k}=1$ for each $(i,j)\in E$, then $\{\hattau{i}{j}{k}:(i,j)\in E\}$ is the minimal non-negative solution of the following linear system:
\begin{equation}   \label{systau}
	\begin{cases}
	\hattau{i}{j}{k} = 0 & \hbox{if } i=k  \\
	\hattau{i}{j}{k} = 1 & \hbox{if }
	 j = k  \hbox{ and } i\neq k  \\
	\hattau{i}{j}{k} = 1 + \sum_{\ell\in V}
	p_{i,j,\ell}\hattau{j}{\ell}{k} & \hbox{if } i\neq k \hbox{ and } j\neq k.
	\end{cases}
	\end{equation}
\end{Theorem}

\begin{proof}
	We first show that $\{\hattau{i}{j}{k}:(i,j)\in E\}$ is a solution of the linear system \eqref{systau}. Let $(i,j)\in E$. If $i=k$ then $\prob{i}{j}{k}{0}{\wt T}{=}=1$ and $\hattau{i}{j}{k}=0$. If $j=k$ then $\prob{i}{j}{k}{1}{\wt T}{=}=1$ and $\hattau{i}{j}{k}=1$. Otherwise, 
	if both $i\neq k$ and $j\neq k$ then,
	by using the definition of expectation we obtain
	\begin{equation*}
	\begin{split}
	\hattau{i}{j}{k}&=\mathbb{E}(\wt T_k\vert X_1=j,X_0=i)\\
	&=2\prob{i}{j}{k}{2}{\wt T}{=}+\sum_{n\geq 3}n\prob{i}{j}{k}{n}{\wt T}{=}.
	\end{split}
	\end{equation*}
Now we observe that $\prob{i}{j}{k}{2}{\wt T}{=}=\prob{i}{j}{2}{k}{X}{=}=p_{i,j,k}$. Then $\hattau{i}{j}{k}$ is given by 
\begin{align*}
	\hattau{i}{j}{k} = &
	\, 2p_{i,j,k}+\sum_{n\geq 3}n\sum_{\ell\neq k}
	\prob{i}{j}{2}{\ell}{X}{=} 
	\P(\wt T_k=n\vert X_2=\ell,X_1=j)   \\
	= & \, 2p_{i,j,k}+\sum_{\ell\neq k}p_{i,j,\ell}
	\sum_{n\geq 3}n \,
	\P(\wt T_k=n\vert X_2=\ell,X_1=j)   \\
	= & \, 2p_{i,j,k}+\sum_{\ell\neq k}p_{i,j,\ell}
	\sum_{m\geq 2}(m+1)
	\prob{j}{\ell}{k}{m}{\wt T}{=}.
\end{align*}
Finally, since $\sum_{m\geq 2}m\prob{j}{\ell}{k}{m}{\wt T}{=}=\hattau{j}{\ell}{k}$ and $\sum_{m\geq 2}\prob{j}{\ell}{k}{m}{\wt T}{=}=1$ for $\ell\neq k$, 
	we have
\begin{align*}
	\hattau{i}{j}{k}
	& = 2p_{i,j,k}+\sum_{\ell\neq k}p_{i,j,\ell}(\hattau{j}{\ell}{k}+1)  \\
	& = 2p_{i,j,k}+\sum_{\ell\neq k}p_{i,j,\ell}\hattau{j}{\ell}{k}+\sum_{\ell\neq k}p_{i,j,\ell}  \\
	& = 1+p_{i,j,k}+\sum_{\ell\neq k}p_{i,j,\ell}\hattau{j}{\ell}{k}  
	 = 1+\sum_{\ell\in V}p_{i,j,\ell}\hattau{j}{\ell}{k},
\end{align*}
where in the last passages we have used the fact that 
$\sum_{\ell\in V}p_{i,j,\ell} = 1$.

Now let $\{y_{i,j}:(i,j)\in E\}$ be another non-negative solution of \eqref{systau}. We show that $y_{i,j}\geq\hattau{i}{j}{k}$ for all $(i,j)\in E$. If $i=k$ or $j=k$ then $y_{i,j}=\hattau{i}{j}{k}$ and the claim follows. 

For $i,j\neq k$ we have $y_{i,j} = 1 + p_{i,j,k} + \sum_{\ell\neq k} p_{i,j,\ell} y_{j,l}$.
Now, let $\ell_0 = i$ and $\ell_1 = j$. 
For any $\bar n \geq 2$ we have
\begin{align*}
    y_{\ell_0,\ell_1} & = 
    1 + p_{\ell_0,\ell_1, k} + \sum_{\ell_2\neq k} 
    p_{\ell_0,\ell_1,\ell_2} y_{\ell_1,\ell_2} \\
    & = 
    1 + p_{\ell_0,\ell_1, k} + \sum_{\ell_2\neq k} 
    p_{\ell_0,\ell_1,\ell_2} 
    \bigg[ 1 + p_{\ell_1,\ell_2, k} + \sum_{\ell_3\neq k} 
    p_{\ell_1,\ell_2,\ell_3} y_{\ell_2, \ell_3} \Big] \\
    & = \ldots \\
    & = \sum_{n = 1}^{\bar n} \Bigg[ 
    \sum_{\ell_2,\ldots,\ell_n \neq k} 
    \prod_{t = 2}^{n} 
    p_{\ell_{t-2},\ell_{t-1},\ell_n} \Bigg]
    \big( 1 + p_{\ell_{n-1},\ell_{n}, k} \big) + 
    \sum_{\ell_2,\ldots,\ell_{\bar n + 1}\neq k} 
    \prod_{t = 2}^{\bar n+1} p_{\ell_{t-2},\ell_{t-1},\ell_{t}} y_{\ell_{t-1},\ell_{t}}  \\
    & \geq \sum_{n = 1}^{\bar n} \Bigg[ 
    \sum_{\ell_2,\ldots,\ell_n \neq k} 
    \prod_{t = 2}^{n} 
    p_{\ell_{t-2},\ell_{t-1},\ell_n} \Bigg]
    \big( 1 + p_{\ell_{n-1},\ell_{n}, k} \big) ,
\end{align*}
where the empty sum (i.e., the square bracket when $n = 1$) 
must be considered equal to $1$
and the last inequality holds since 
$y_{\ell_{t-1},\ell_{t}}\geq 0$.
Rearranging terms, 
\begin{align*}
    y_{\ell_0,\ell_1} & \geq 
    1 + \sum_{\ell_2} p_{\ell_0,\ell_1,\ell_2} + \sum_{n=3}^{\bar n} 
    \sum_{\ell_2\neq k}\cdots
    \sum_{\ell_{n-1} \neq k} \sum_{\ell_n}
    \prod_{t = 2}^{n} 
    p_{\ell_{t-2},\ell_{t-1},\ell_n}  \\
    & = \sum_{n = 1}^{\bar n} \mathbb{P}(\wt T_k \geq n
    | X_1 = j, X_0 = i) \\
    & \geq \sum_{n = 1}^{\bar n} n \,
    \mathbb{P}(\wt T_k = n | X_1 = j, X_0 = i).
\end{align*}
Finally, passing to the limit $\bar n\to\infty$ we obtain
\begin{align*}
    y_{i,j} 
    & \geq \sum_{n = 1}^{\infty} n \,
    \mathbb{P}(\wt T_k = n
    | X_1 = j, X_0 = i)   = \mathbb{E}(\wt T_k | X_1 = j, X_0 = i) 
    = \hat \tau_{i,j\to k} ,
\end{align*}
and the proof is complete.
\end{proof}

Analogously to the argument of Corollary \ref{cor:hitprob}, we obtain the following characterization for the second-order hitting times directly from Equation \eqref{eq:hitt2o}.

\begin{Corollary}    \label{cor:hittim2o}
For any fixed $k\in V$, the vector $(\wt\tau_{i\to k})_{i\in V}$ coincides with the vector $(\hattau{i}{j}{k})_{(i,j)\in E}$ pre-multiplied by  the matrix $M$ in \eqref{eq:M}.
\end{Corollary}

\subsection{An alternative approach}   \label{sec:alternative}

The linear system in Theorem \ref{theotau} allows us to compute second-order mean hitting times by means of the transition probabilities of the stochastic process $\{X_n,n\in\N\}$. 
The same quantities can be obtained using the fact that to any second-order random walk $\{X_n\}$ on $\G$,  corresponds  a suitable random walk  $\{W_n\}$ in $\wh\G$, as in \eqref{eq:defW}. The two walks have the same transition probability, but the second is shorter than the first by one step. By correcting this discrepancy, we can obtain mean hitting times for $\{X_n\}$ from those for $\{W_n\}$.

To this end, for any fixed subset $S\subset E$, 
consider the random variables
$$
   \wh T_S = \min\{n \geq 0 : W_n  \in S\} ,
   \qquad
   \wh T^+_S  = \min\{n \geq 1 : W_n \in S\} ,
$$ 
and the corresponding expectations
\begin{equation}   \label{eq:hattau}
      \wh\tau_{e\to S} = \mathbb{E}(T_S | W_0 = e) ,
   \qquad
   \wh\tau^+_{e\to S} = \mathbb{E}(T^+_S | W_0 = e) . 
\end{equation}
These quantities can be computed by means of the techniques recalled in Section \ref{sect:rw}, since $\{W_n\}$ is a classical Markov chain.
The theorem below shows how mean hitting \eqref{eq:hattau} for $\{W_n\}$ relate to the solution of the linear system in \eqref{systau}. 
It follows that the problem of computing second-order mean hitting times in $\G$ can be expressed as the problem of computing traditional mean hitting times in $\wh\G$. 

Let us consider a node $k\in V$. By Theorem \ref{theomarkovtau}, $\{\hat\tau_{e\to\inp_k}: e\in E\}$ is the minimal non-negative solution of the following linear system:
\begin{equation}   \label{linsysA}
   \begin{cases}
   \hat\tau_{e\to\inp_k}=0 & \hbox{if } e\in\inp_k \\
   \hat\tau_{e\to\inp_k}=1+\sum_{f\in E}\widehat{P}_{ef}\hat\tau_{f\to\inp_k} & \hbox{if $e\notin\inp_k$.}
   \end{cases}
\end{equation}

\begin{Theorem}   \label{theo_equivalence}
In the previous notation, it holds
$$
	\hattau{i}{j}{k} = \begin{cases}
	0 & \hbox{if } i=k     \\
	\hat\tau_{(i,j)\to\inp_k} + 1 & 
	\hbox{if } i\neq k .
	\end{cases}
$$
\end{Theorem}
\begin{proof}
We show that each non-negative solution of \eqref{linsysA} corresponds to a non-negative solution of \eqref{systau} and vice-versa. 
Let $\{z_{i,j}:(i,j)\in E\}$ be a non-negative solution of \eqref{linsysA}. For all $(i,j)\in E$, we define $y_{i,j}$ as
$$
	y_{i,j}=
	\begin{cases}
	0 & \hbox{if } i = k \\
	z_{i,j}+1 & \hbox{if } i\neq k .
	\end{cases}
$$
Then $\{y_{i,j}:(i,j)\in E\}$ is a non-negative solution of \eqref{systau}. In fact, if $j=k$ but $i\neq k$ then $y_{i,j}=1$. 
On the other hand, when $i\neq k$ and $j\neq k$
we have
\begin{align*}
	y_{i,j} = 1+z_{i,j} & = 
	1 + \Bigl(1+\sum_{\ell\in V}
	\widehat{P}_{(i,j)(j,\ell)}z_{j,\ell}\Bigr)  \\
	& = 1 + \sum_{\ell\in V} \widehat{P}_{(i,j)(j,\ell)}
	(1+z_{\ell,m}) 
    = 1 + \sum_{\ell\in V}
	p_{i,j,\ell} \, y_{j,\ell} .
\end{align*}
In the previous passages we used the fact that $\wh P$ is stochastic and $\wh P_{(i,j)(\ell,m)} = 0$ if $\ell \neq j$. 
Conversely, if $\{y_{i,j}:(i,j)\in E\}$ is a non-negative solution of \eqref{systau}, then $y_{i,j}\geq 1$ if $i\neq k$. Thus, if we put, for all $(i,j)\in E$,
$$
	z_{i,j} = \begin{cases}
	0 & \hbox{if } i = k   \\
	y_{i,j}-1 & \hbox{if } i\neq k ,
	\end{cases}
$$
then $\{z_{i,j}:(i,j)\in E\}$ is a non-negative solution of \eqref{linsysA}.
The thesis follows by considering the minimal non-negative solution of both systems.
\end{proof}

\subsection{Second-order mean return times}  \label{sec:returntimes}

In this section, we formulate the concept of return time for second-order random walks and we give an explicit formula for their computation in 
arbitrary networks, whenever the numbers $\wt\tau_{i,j\to k}$ can be computed from \eqref{systau}. 
For a set $S\subset V$, consider the random variable 
\[
   \wt T^+_S = \min\{n \geq 1 : X_n \in S\} . 
\]
We are interested in computing the \emph{second-order mean return time} to $S$, namely, $\wt\tau_S=\mathbb{E}(\wt T^+_S \vert X_0 \in S)$. This quantity represents the average number of steps that are required to reach again $S$, starting from a node in $S$ and moving through the vertices of $\G$ according to a second-order random walk. For simplicity, in this section we limit ourselves to the single-node case $S = \{i\}$. We will go back to the general case in Section \ref{sec:final}.
Firstly, we observe that
$$
   \wt\tau_i=
   \mathbb{E}(\wt T^+_i\vert X_0=i)
   = \sum_{j\in V} \P(X_1=j\vert X_0=i)
   \mathbb{E}(\wt T^+_i\vert X_1=j,X_0=i) .
$$
Introducing the auxiliary notation 
$\wt\tau_{i,j}= \mathbb{E}(\wt T^+_i\vert X_1=j,X_0=i)$, we obtain 
\begin{equation}    \label{eq:2oret}
    \wt\tau_i = \sum_j p'_{i,j} \wt\tau_{i,j} . 
\end{equation}
The intuition behind this formula is that the random walker chooses at first an out-neighbour $j$ of $i$ at random and then follows a second-order random walk that ends with node $i$. 
In fact, $\wt\tau_{i,j}$ can be easily computed by using the vector $\{\wt\tau_{i,j\to k}:(i,j)\in E\}$ obtained from \eqref{systau}. Indeed, it holds
\begin{align*}
   \wt\tau_{i,j} & = 
   \sum_{k\in V} \P(X_2=k \vert X_1=j,X_0=i)
   \mathbb{E}(\wt T^+_i\vert
   X_2 = k,X_1=j,X_0=i)  \\
   & = \sum_{k\in V}p_{i,j,k} \bigl( 1 + 
   \mathbb{E}(\wt T_i\vert 
   X_1 = k,X_0 = j) \bigr) \\
   & = 1 + \sum_{k\in V}p_{i,j,k}\wt\tau_{j,k\to i}.
\end{align*} 
Placing the formula above into \eqref{eq:2oret}, yields
$$
   \wt\tau_i = 1 + \sum_{j,k\in V}
   p'_{i,j} p_{i,j,k}\wt\tau_{j,k\to i}. 
$$
Apparently, this formula does not admit a simple matrix-vector form as in Corollary \ref{cor:hittim2o}.
In the sequel, we propose to adopt a specific choice for the numbers $p'_{i,j}$ that allows us to obtain a result for any $S\subset V$ analogous to Lemma \ref{lem:kac}.

\section{The pullback of a second-order random walk}   \label{sec:main}

From now on, we suppose that the Markov chain associated to $\wh P$ is ergodic, so that there exists a (unique) vector $\wh \pi > 0$ with $\wh\pi^T \wh P = \wh\pi^T$. In this case, we can consider the additional assumption that the random walk $\{W_n\}$ on $\wh\G$ is at equilibrium, i.e., $\mathbb{P}(W_n = e) = \wh\pi_e$, for $n \geq 0$.
Indeed, if the Markov chain $\{W_n\}$ 
is at the stationary state, then,
owing to the correspondence between the events $X_{n+1} = i$ and $W_n \in \inp_i$ for $n \geq 0$, the analysis of the process $\{X_n\}$ is simplified and reveals new properties.

For every $e\in E$, define
\begin{equation}   \label{eq:lambdae}
   \lambda_e = \frac{\wh\pi_e}{\sum_{f:\ter(f) = \ter(e)}\wh \pi_f} . 
\end{equation}
The numbers $\lambda_e$ appearing in \eqref{eq:lambdae}
have a simple probabilistic interpretation, namely $\lambda_e = \mathbb{P}(W_n=e | W_n\in \inp_{\ter(e)})$ is 
the conditional probability that the random walker in $\wh\G$ is in $e$, given that they are also in $\inp_{\ter(e)}$.
Moreover, let $L\in\R^{|V|\times |E|}$ be the ``lifting'' matrix defined as follows:
\begin{equation}   \label{eq:L}
   L_{ie} = \begin{cases}
   \lambda_e & \ter(e) = i \\ 0 & \hbox{else,}
   \end{cases}
\end{equation}
Note that the scalars $\lambda_e$ fulfill the condition
$$
   \sum_{e\in\inp_i} \lambda_e = 1 , \qquad i = 1,\ldots, n .
$$
This condition implies that $L$ is stochastic, $L\uno = \uno$.
If $p$ is a probability vector on the nodes of $\G$, then 
$\wh p^T = p^TL$ is a probability vector defined on the edges, and
fulfills the identity $p_i = \sum_{e\in\inp_i} \wh p_e$. 
Furthermore, let $R\in\R^{|E|\times |V|}$ be the ``restriction'' matrix 
\begin{equation}   \label{eq:R}
   R_{ei} = \begin{cases}
   1 & \hbox{if } \ter(e) = i \\ 0 & \hbox{else.} \end{cases}
\end{equation}
Given any probability distribution $\wh p$ on the edges of $\G$, 
the product $p^T = \wh p^T R$ defines a probability distribution on the nodes
such that $p_i = \sum_{e\in\inp_i} \wh p_e$. 
We note in passing that the product $LR$ is the identity in $\R^n$.

\begin{Theorem}   \label{thm:main}
Let $P = L\wh P R\in\R^{n\times n}$, 
\begin{equation}    \label{eq:pullback}
   P_{ij} = \sum_{e\in \inp_i}\sum_{f\in \inp_j} \lambda_e \wh P_{ef} .
\end{equation}
Then, $P$ is an irreducible stochastic matrix. 
Assuming that the Markov chain $\{W_n\}$ is at equilibrium with stationary density $\wh\pi$, then also $\{X_n\}$ is at equilibrium with stationary density $\pi^T = \wh \pi^T R$ and, for $n \geq 1$,
$$
   P_{ij} = \mathbb{P}(X_{n+1} = j | X_n = i ) .
$$
Moreover, the initial transition probabilities \eqref{eq:X01} are 
\begin{equation}   \label{eq:initial}
   p'_{i,j} = \frac{\wh\pi_{(i,j)}}{\sum_{f\in\out_i}\wh \pi_f} , \qquad (i,j)\in E,
\end{equation}
and the vector $\pi^T = \wh \pi^T R$ is the (unique) invariant vector
of the Markov chain associated to $P$.
\end{Theorem}

\begin{proof}
First, note that the matrix $P$ is non-negative and stochastic. 
To prove that $P$ is irreducible, let $i,j\in V$ be arbitrary. Owing to the irreducibility of $\wh P$, for any $e\in\inp_i$ and $f\in\inp_j$
there is a path in $\wh\G$ from $e$ to $f$. Let $e = e_0, e_1, \ldots, e_m = f$ be such path. Then, the sequence $\ter(e_0), \ter(e_1),\ldots,\ter(e_m)$ is a path in $\G$ from $i$ to $j$. 

Thus, owing to the correspondence between $X_{n+1} = i$ and $W_n \in \inp_i$, we have 
\begin{align*}
    \mathbb{P}( X_{n+1} = j | X_n = i ) & = 
    \mathbb{P}( W_{n} \in\inp_j | W_{n-1} \in\inp_i ) \\
    & = \sum_{e\in\inp_i}\sum_{f\in\inp_j} 
    \mathbb{P} ( W_{n} = f | W_{n-1} = e)
    \mathbb{P}(W_{n-1} = e | W_{n-1} \in\inp_i) \\ 
    & = \sum_{e\in\inp_i}\sum_{f\in\inp_j} 
    \wh P_{ef} \lambda_e  = P_{ij} 
\end{align*}
for any $n \geq 1$. Moreover, consider the vector $\pi^T = \wh\pi^T R$. In detail,
\begin{equation}   \label{eq:pi_i}
   \pi_i = \sum_{e\in\inp_i}\wh\pi_e > 0 , \qquad i = 1,\ldots, n.
\end{equation}  
To prove that $\pi$ is an 
invariant vector of $P$, first observe that $\wh\pi^T RL = \wh\pi$. Indeed, owing to \eqref{eq:lambdae},
for any fixed $e\in E$  we obtain
$$
   (\wh\pi^T RL)_e = \lambda_e \sum_{f:\ter(f)=\ter(e)}\wh\pi_f
   = \wh\pi_e .
$$
Moreover, $\pi^T\uno = \wh\pi^T R\uno = \wh\pi^T\uno = 1$, eventually yielding
$$
   \pi^T P = \wh \pi^T R L\wh P R = \wh\pi^T \wh P R = \wh \pi^T R = 
   \pi^T .
$$
To complete the proof it is sufficient to observe that, if $\{W_n\}$ is at equilibrium 
then also $\{X_n\}$ is at equilibrium and
\begin{align*}
   \wh\pi_{(i,j)} = \mathbb{P}(W_0 = (i,j))
   & = \mathbb{P}(X_1 = j, X_0 = i) \\
   & = \mathbb{P}(X_1 = j | X_0 = i) 
   \mathbb{P}(X_0 = i)
   = p'_{i,j} \sum_{f\in\inp_i} \wh\pi_f . 
\end{align*}
Thus, $p'_{i,j} = \wh\pi_{(i,j)}/\sum_{f\in\inp_i} \wh\pi_f$. Finally, as $\sum_{f\in\inp_i} \wh\pi_f = \sum_{f\in\out_i} \wh\pi_f$, we obtain  \eqref{eq:initial}, which concludes the proof.
\end{proof}

We call the matrix $P$ in \eqref{eq:pullback} the \textit{pullback} of $\wh P$ on $\G$. Note that these two matrices correspond to two different stochastic processes, as briefly pointed out in Remark \ref{rem:pullback} below. Moreover, as we may expect, we remark that  the pullback operation is not injective, as different stochastic processes on $\wh \G$ may have the same pullback in $\G$. This is shown by the Example \ref{ex:example-pullback} below, where the classic and non-backtracking processes are shown to have the same pullback.

\begin{Remark}  \label{rem:pullback}
As shown in Theorem \ref{thm:main}, the entries of the pullback matrix $P$ correspond to first-order transition probabilities of the stochastic process $\{X_n\}$. However, this does not mean that $\{X_n\}$ is a Markov chain, which would be true if (and only if) the sequence $\{X_n\}$ were to satisfy the Markov condition
$\mathbb{P}(X_{n+1} = k | X_n = j)
= \mathbb{P}(X_{n+1} = k | X_n = j, X_{n-1} = i)$,
for every $(i,j),(j,k)\in E$, at least under the considered assumptions. On the other hand, Equation \eqref{eq:pullback} yields 
$$
   P_{ij} = \sum_{(h,i)\in E} \lambda_{(h,i)} p_{h,i,j} ,
$$
that is, $P_{ij}$ is a convex combinations of the numbers $p_{h,i,j}$.
Consequently, the Markov condition holds true if and only if the second-order transition probabilities \eqref{eq:X} are independent of the oldest state, in which case $\{X_n\}$ is trivially a Markov chain.

Actually, the construction of $P$ from $\wh P$ is reminiscent of a so-called \emph{lumping} of $\{W_n\}$, see e.g., \cite{WeakLump} and   \cite[Sec.~6.3-4]{KemenySnell}. Lumping methods are known to allow the construction of Markov chains with a reduced number of states from certain finite Markov chains with a larger states set. The corresponding
transition matrices are related by an identity similar to the one defining $P$ from $\wh P$ in Theorem \ref{thm:main}. 
However, the process $\{X_n\}$ is not a lumping of $\{W_n\}$ since it is not a Markov chain, even when at equilibrium. This is the reason why second-order mean hitting times differ ostensibly from mean hitting times corresponding to $P$, as shown in various numerical experiments in Section \ref{experiments}.
\end{Remark}

\begin{example}   \label{ex:example-pullback}
Let $\G$ be undirected and (strongly) connected, with adjacency matrix $A$, and let $\wh \G$ be its line graph. Consider the random walk on $\wh\G$ defined as in Example \ref{ex:classic}, with transition matrix $\wh P$.  It is not hard to verify that the pullback of $\wh P$ on $\G$ coincides with the transition matrix of the standard random walk on $\G$.
The same conclusion is true also for the matrix $\wh P$ considered 
in Example \ref{ex:nbtrw}. Indeed, if $A_{ij} = 0$ then
$$
   (L\wh PR)_{ij} = 
   \sum_{e\in\inp_i} \sum_{f\in\inp_j} \lambda_e \wh P_{ef} = 0
$$
because the condition $\sou(f) = \ter(e)$ is always false.
On the other hand, if $A_{ij} = 1$ then
\begin{align*}
   (L\wh PR)_{ij} 
   & = \frac{1}{\sum_{e\in\inp_i}} \sum_{f\in\inp_j} 
   \lambda_e \wh P_{ef} \\
   & = \sum_{e\in\inp_i} \sum_{f\in\inp_j} \frac{\lambda_e}{d^+_{\sou(f)} - 1}   \\
   & = (|\inp_i| -1) \frac{1/d^-_i}{d^+_i - 1} = \frac{1}{d_i} , 
\end{align*}
since $|\inp_i| = d^-_i = d^+_i = d_i$.
By linearity, also the pullback of the backtrack-downweighted random walk in Example \ref{ex:penalized} coincides with the standard random walk on $\G$.
\end{example}

\subsection{Hitting and return times at equilibrium}   \label{sec:final}

Due to Theorem \ref{thm:main}, it is convenient to assume that the initial probabilities of the second-order stochastic process under consideration \eqref{eq:X01} are equal to $\lambda_{(i,j)}$. Hence, from here onward we consider the ``equilibrium assumption''
$\mathbb{P}(X_0 = i) = \pi_i$ together with \eqref{eq:initial}.
With this assumption, the presentation of the mean 
hitting and return times introduced in the previous section can be 
carried out with greater clarity and simplicity.
To begin with, we derive an algebraic expression for the second-order hitting times matrix $\wt T = (\wt\tau_{i\to j})$ of a very different nature from that at the basis of Equation \eqref{eq:hitt2o}.  

\begin{Theorem}   \label{thm:tildeT}
Let $\wt T = (\wt\tau_{i\to j})_{i,j\in V}$ and $\wh T = (\wh\tau_{e\to f})_{e,f\in E}$. There exists a vector $a = (\alpha_e)_{e\in E}$ and a vector $b = (\beta_i)_{i\in V}$ such that 
$$
   \wt T = L \wh T  \, \mathrm{Diag}(a) R - \uno b^T ,
$$
where $L$ is as in \eqref{eq:L} and $R$ is as in \eqref{eq:R}.
\end{Theorem}

\begin{proof}
Introduce the matrix $C\in\R^{|E|\times|V|}$ such that $C_{ei} = \wh\tau_{e\to \inp_i}$ is defined as in Section \ref{sec:alternative}.
From the identity
\begin{align*}
   \wt\tau_{i\to j} = \mathbb{E}(\wt T_j | X_0 = i) 
   & = \mathbb{E}(\wh T_{\inp_j} | W_0 \in\inp_i) \\ & = \sum_{e\in\inp_i}
   \mathbb{P}(W_0 = e \vert W_0\in\inp_i) \, \mathbb{E}(\wh T_{\inp_j} | W_0 = e)
   = \sum_{e\in\inp_i} \lambda_e \wh\tau_{e\to\inp_j}
\end{align*}
it follows that $\wt T = LC$.
Using the technical Lemma \ref{lem:1}, for every $e\in E$ and $i\in V$ there exist coefficients $\alpha_e$ and $\beta_{i}$ such that $\wh\tau_{e\to \inp_i} = \sum_{f\in\inp_i} \alpha_f \wh\tau_{e\to f} - \beta_i$. The latter identity can be expressed in matrix terms as
$$
   C = \wh T \, \mathrm{Diag}(a) R - \uno b^T .
$$
Finally, as  $L\uno = \uno$, we conclude.
\end{proof}

The previous theorem yields a formula for the matrix $\wt T$ that is amenable to numerical computations, once the matrix $\wh T$ has been computed by standard methods via, e.g., \eqref{eq:Tmatrix}.
The next result proves that second-order mean return times defined in Section \ref{sec:returntimes} coincide with the usual return times of the random walk associated with the pullback of $\wh P$ on $\G$.

\begin{Theorem}    \label{thm:2oKac}
Let $\pi$ be the stationary density of the pullback of $\wh P$ on $\G$. For every $S\subset V$ it holds $\wt\tau_S = 1/\sum_{i\in S}\pi_i$.
In particular, $\wt\tau_i = 1/\pi_i$ for $i\in V$.
\end{Theorem}

\begin{proof}
Let $\wh S = \cup_{i\in S} \inp_i$.
Then, $\wt\tau_S$ can be reformulated as the mean return time in $\wh S$ of the chain $\{W_n\}$.
From Lemma \ref{lem:kac} and \eqref{eq:pi_i}, we get
$$
   \wt\tau_{S} = 
   \mathbb{E}(\wh T^+_{\wh S} | W_0 \in \wh S)
   = \frac{1}{\sum_{i\in S}\sum_{e\in \inp_i} \wh\pi_e} 
   = \frac{1}{\sum_{i\in S}\pi_i} ,
$$
and we have the claim.
\end{proof}

\begin{example}    \label{ex:last}
Let $\G = (V,E)$ be finite and undirected. For $\alpha\in [0,1]$, the stationary density $\wh\pi$ of the matrix $\wh P^{(\alpha)}$ in Example \ref{ex:penalized} is a constant vector, $\wh\pi_e = 1/|E|$, as the transition matrix is bi-stochastic. Consequently, the second-order mean return times coincides with the mean return times of the uniform random walk on $\G$, namely, $\wt\tau_i = (\sum_j d_j)/d_i$, independently of $\alpha$ and the topology of $\G$. 
Moreover, in this example Equation \eqref{eq:initial} becomes  $\lambda_{(i,j)} = 1/d_i$ being $\G$ undirected, i.e., the first step in the second-order process coincides with one step of a standard uniform random walk on $\G$.
\end{example}

Theorem \ref{thm:2oKac} yields a second-order analogous to Kac's lemma, see Lemma \ref{lem:kac}. For a family of networks that exhibit some specific symmetries, Lemma \ref{lem:Kemeny} can also be extended to second-order random walks, as shown in the sequel.

The coefficients $\alpha_e$ appearing in the proof of Theorem \ref{thm:tildeT} come from Lemma \ref{lem:1} and, according to \eqref{eq:alpha}, can be expressed as $\alpha_e = \wh\pi_e\wh\tau^+_{e\to\inp_i}$ where 
$i = \ter(e)$ and $\wh\tau^+_{e\to\inp_i}$ represents the mean return time to $\inp_i$ of $\{W_n\}$,  after leaving it from $e\in\inp_i$, as defined in \eqref{eq:hattau}.
Suppose that these return times depend on $i$ but not on $e\in\inp_i$. This assumption is clearly verified if for every $e,f\in\inp_i$ there is an automorphism of $\wh\G$ that exchanges $e$ and $f$, i.e., the nodes $e$ and $f$ in $\wh \G$ can be exchanged without altering the topology of $\wh \G$. 
The latter condition is satisfied if, for example, $\{W_n\}$ is a uniform random walk on a bipartite regular graph, that is, an undirected, bipartite graph whose adjacency matrix has the form
$$
   \begin{pmatrix} O & \uno\uno^T \\
   \uno \uno^T & O \end{pmatrix} ,
$$
in which the diagonal blocks can have different sizes.

For notational simplicity, denote by $\rho_i = \wh\tau^+_{e\to\inp_i}$ that common value. Then, from Theorem \ref{thm:main} we have
$$
   1 = \sum_{e\in\inp_i} \alpha_e = 
   \sum_{e\in\inp_i} \wh\pi_e\wh\tau^+_{e\to\inp_i}
   = \rho_i \sum_{e\in\inp_i} \wh\pi_e 
   = \rho_i \pi_i .
$$
Thus $\rho_i = 1/\pi_i$. Moreover, for $e\in\inp_i$ we obtain
$$
   \alpha_e = \wh\pi_e \rho_i = \wh\pi_e / \pi_i .
$$
Hence, in the notation of Theorem \ref{thm:tildeT}, we have that
$\mathrm{Diag}(a) R \pi = \wh\pi$,
and thus
$$
   \wt T\pi = L\wh T \, \mathrm{Diag}(a) R \pi
   - \uno b^T\pi
   = L\wh T \wh\pi - \uno b^T\pi ,
$$
where $b = (\beta_i)_{i\in V}$ is the coefficient vector defined in Theorem \ref{thm:tildeT}. By Lemma \ref{lem:Kemeny}, we have $\wh T \wh\pi = \kappa \uno$ where $\kappa$ is the Kemeny's constant of $\wh T$, and we eventually obtain: 
$$
   \wt T\pi = \kappa L \uno - \uno b^T\pi
   = \uno (\kappa - b^T\pi) .
$$
Combining all the observations above, we obtain the following second-order equivalent of the random target lemma:

\begin{Corollary}
Let $\pi$ be the stationary density given in Theorem \ref{thm:main}. In the equilibrium assumption discussed at the beginning of Section \ref{sec:final} and using
the previous notation, we have that, 
if for every $i\in V$ the return time $\wh\tau^+_{e\to\inp_i}$ depends on $i$ but not on $e\in\inp_i$, then 
$$
   \sum_{j\in V} \pi_j \wt\tau_{i\to j} = \wt\kappa ,
$$
independently on $i$, where $\wt\kappa = \kappa - b^T\pi$.
\end{Corollary}

\section{Numerical experiments on real-world networks}\label{experiments}

In this section, we present numerical results on the computation of second-order mean hitting times on some real-world networks.
We consider three undirected and unweighted networks: \texttt{Guppy} \cite{guppy}, that represents the social interactions of a population of  free-ranging
guppies (\emph{Poecilia reticulata}), \texttt{Dolphins} \cite{dolphins}, that represents the social structure of a group of dolphins, and \texttt{Householder93} \cite{Moler}, that represents a collaboration network of a group of mathematicians. 
Since these graphs contain dangling edges, we remove in sequence all the nodes of degree $1$ until each vertex has degree greater or equal to $2$. Table \ref{structure_networks} shows the dimension of these networks before and after removing such leaf nodes. 
Our experiments illustrate various statistics on the second-order hitting times $\wh\tau_{i\to j}$ for non-backtracking random walks in these graphs. Note that hitting times computed with the pullback transition matrix coincide with the hitting times $\tau_{i\to j}$ for the standard random walk, see Example \ref{ex:last}.

\begin{table} 
    \centering
    \begin{tabular}{l c c c c c c c}
    \toprule 
    & \multicolumn{3}{c}{\textbf{before the removal}} & &  \multicolumn{3}{c}{\textbf{after the removal}}\\
   \cline{2-4}\cline{6-8}
         & nodes & edges & diameter &  & nodes & edges & diameter \\ 
         \midrule
        \texttt{Guppy}    & 99 & 726 & 6 & & 98 & 725 & 5 \\
        \texttt{Dolphins} & 62 & 159 & 8 & & 53 & 150 & 7 \\
        \texttt{Householder93} & 104 & 211 & 7 &  &  73 & 180 & 5 \\
\bottomrule
    \end{tabular} \vspace{.3em}
      \caption{Dimension of the networks used in our experiments
   before and after the removal of leaf nodes.
   \label{structure_networks}}
\end{table}

In Figure \ref{fig:1} we scatter plot the average non-backtracking mean hitting times $\widehat{m}_j = \sum_{i\in V} \wh\tau_{i\to j}/|V|$ against the average traditional mean hitting times $m_j = \sum_{i\in V} \tau_{i\to j}/|V|$.
The dotted red line is the bisector of the first quadrant.
We observe that, in all these networks, the ratio between $m_j$ and $\widehat{m}_j$ is always smaller than $1$. Intuitively, this means that, on average, non-backtracking random walks are shorter than classical random walks.

\begin{figure}
     \centerline{\includegraphics[width=1\linewidth]{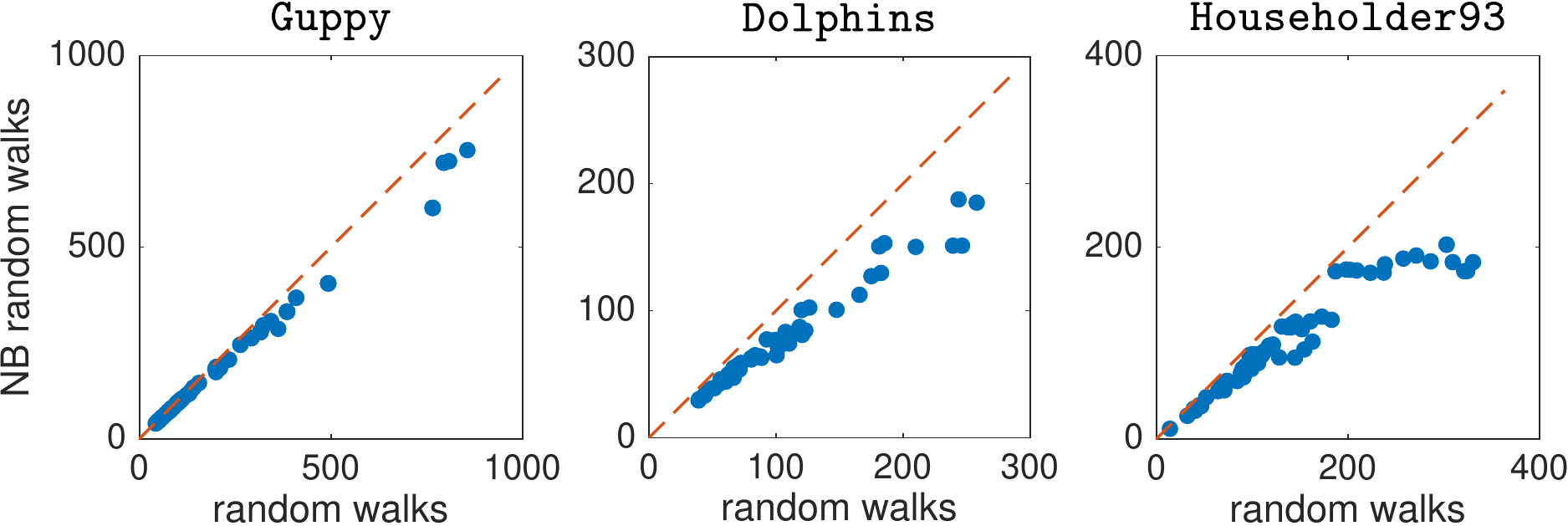}}
     \caption{Scatter plot of the average non-backtracking mean hitting times against the average traditional mean hitting times for the networks \texttt{Guppy}, \texttt{Dolphins} and \texttt{Householder93}. The dotted lines represent the bisectors.}
     \label{fig:1}
\end{figure}

Figure \ref{fig:2} displays the value of the non-backtracking mean access time 
\begin{equation}   \label{eq:ai}
   a_i = \sum_{j\in V} \pi_j \wh\tau_{i\to j}
\end{equation}
with respect to the node index $i\in V$. Recall that these values for the standard random walks are all equal to the Kemeny constant of the graph, see Lemma \ref{lem:Kemeny}. Remarkably, the values for non-backtracking walks have a very narrow range, as shown by the small variation of the values in the $y$-axes.  This experiment shows that in these graphs the claim of Corollary \ref{cor:hittim2o} is approximately valid, even though the assumptions about the structure of the graph are not satisfied.

\begin{figure}
     \centerline{\includegraphics[width=1\linewidth]{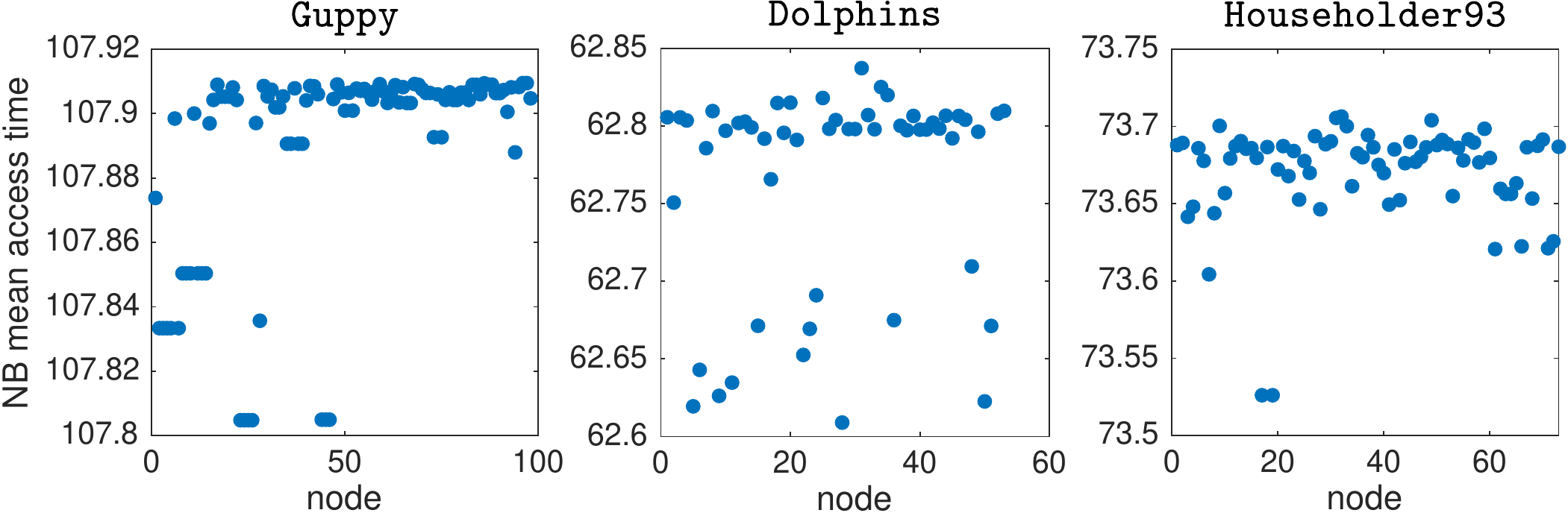}}
     \caption{Plot of the non-backtracking mean access times \eqref{eq:ai} for the networks \texttt{Guppy}, \texttt{Dolphins} and \texttt{Householder93}.}
     \label{fig:2}
 \end{figure}

Finally, we computed mean hitting times for the backtrack-downweighted random walks introduced in Example \ref{ex:penalized}, with varying $\alpha\in[0,1]$. The corresponding stochastic processes interpolate between the standard random walk ($\alpha = 1$) and the non-backtracking variant ($\alpha = 0$). Nevertheless, the pullback matrix does not depend on $\alpha$, as noted in Example \ref{ex:example-pullback}. 
Let $\tau^{(\alpha)}_{i\to j}$ denote second-order mean hitting times associated with the transition matrix $\wh P^{(\alpha)}$.
To better illustrate the effect of penalizing backtracking steps, we measure the mean hitting times ratio
\begin{equation}   \label{eq:alpha_ratios}
   r^{(\alpha)}_j = 
   \frac{\sum_{i\in V} \tau^{(\alpha)}_{i\to j}}{\sum_{i\in V}\tau^{(1)}_{i\to j}} .
\end{equation}
The curves in Figure \ref{fig:3} illustrate the behavior of the average value (magenta line), along with the maximum and the minimum value  (black lines) of these ratios. The values for $\alpha = 0$ correspond to the results shown in Figure \ref{fig:1}.

\begin{figure}
     \centerline{\includegraphics[width=1\linewidth]{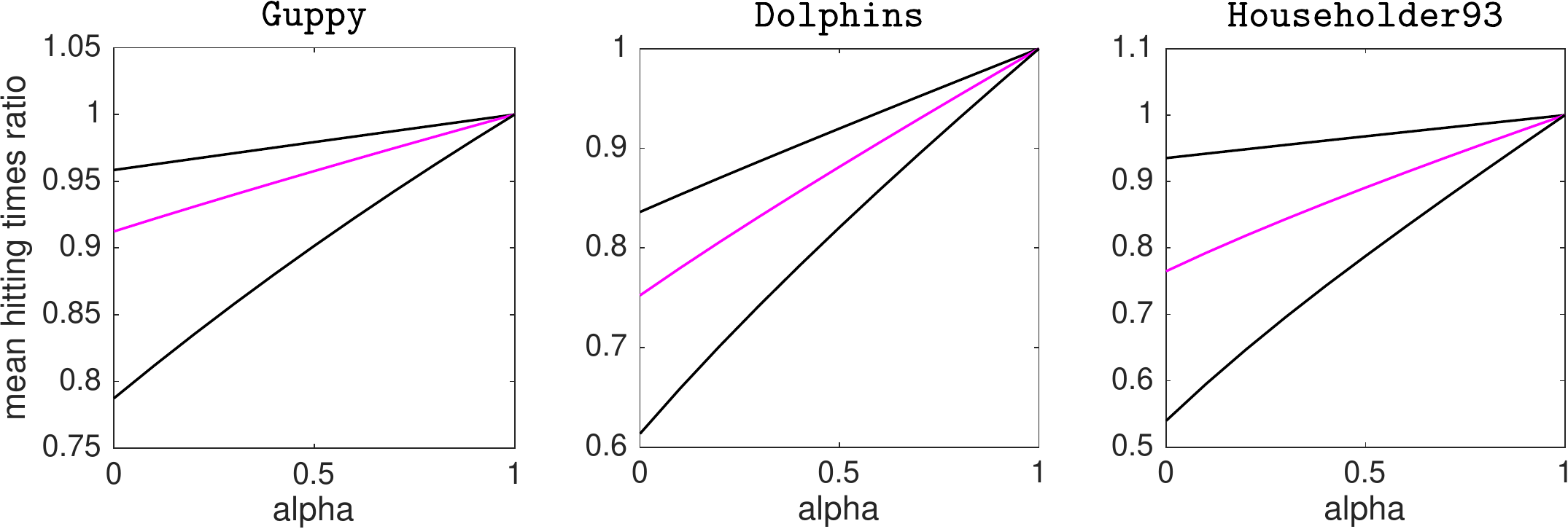}}
     \caption{Behavior of the mean hitting time ratios \eqref{eq:alpha_ratios} for the backtrack-penalized random walks. The solid lines represent maximum, average, and minimum values of the mean hitting time ratios, respectively,  as $\alpha$ varies in $[0,1]$.}
     \label{fig:3}
 \end{figure}


\section{Conclusions}

Second-order random walks, which include non-backtracking random walks as a particular case, provide stochastic models for navigation and diffusion processes on networks that can offer improved modeling capacity with respect to classical random-walks, and thus allow us to better capture real-world dynamics on networks. In this paper, we propose a very general framework to properly define mean hitting times and mean return times for second-order random walks for directed graphs with finite or countably many nodes.
Our approach is based on the correspondence between the second-order random walk and a standard random walk obtained by first lifting the second-random walk to the directed line graph and then pulling back onto the original graph. The ``pullback'' Markov chain obtained this way has the same stationary density as the initial second-order random process and allow us to transfer several well-known results for standard random walks to the second-order case. For example, we prove that second-order mean return times coincide with the mean return times of the corresponding pullback Markov chain and we show how to extend the Kac's and random target lemmas to second-order random walks.

These results prompt us to further investigations concerning computational issues as well as theoretical properties of second-order random walks. 
For instance, it is certainly interesting to devise more efficient numerical methods to
compute second-order hitting times
than the one based on  
Corollary \ref{cor:hittim2o} and Theorem \ref{thm:tildeT}.
Furthermore, 
numerical experiments show 
that non-backtracking hitting times are consistently smaller than their corresponding classical, backtracking versions. 
In fact, reducing the probability of backtracking results in a reduction of the mean hitting times. However, a formal explanation of this fact, at the best of our knowledge, is still missing.

\section*{Acknowledgement}
The work of Dario Fasino and Arianna Tonetto has been carried out in the framework of the departmental research project `ICON: Innovative Combinatorial Optimization in Networks', Department of Mathematics, Computer Science and Physics (PRID 2017), University of Udine, Italy. The work of Dario Fasino and Francesco Tudisco has been partly supported by Istituto Nazionale di Alta Matematica, INdAM-GNCS, Italy.




\end{document}